\documentclass[12pt]{article}
\usepackage{a4}
\usepackage{amsthm}
\usepackage{amsmath}
\usepackage{amsfonts}
\usepackage{cite}
\usepackage{epsfig}
\usepackage{graphicx}

\usepackage{booktabs}

\usepackage{hyperref}
\newcommand{\footremember}[2]{%
    \footnote{#2}
    \newcounter{#1}
    \setcounter{#1}{\value{footnote}}%
}
\newcommand{\footrecall}[1]{%
    \footnotemark[\value{#1}]%
}

\newtheorem{conjecture}{Conjecture}
\newtheorem{theorem}{Theorem}

\newtheorem{corollary}[theorem]{Corollary}
\newcommand\NN{{\mathbb N}}
\newcommand\eexp{\mbox{e}}
\newcommand\Tr{\mbox{Tr}\;}

\newcommand{\matbd}[1]{\ensuremath{\boldsymbol{#1}}}  
\newcommand{\mA}{\matbd{A}}
\newcommand{\vu}{\ensuremath{\textbf{u}}}

\begin{document}
\title{Walk entropy and walk-regularity}

\author{
        Kyle Kloster\footremember{NCSU}{Department of Computer Science, NC State University, Raleigh, NC 27695, USA. E-mail: {\tt \{kakloste,blair\_sullivan\}@ncsu.edu}. The work of this author was supported in part by the Gordon \& Betty Moore Foundation's Data-Driven Discovery Initiative through Grant GBMF4560 to Blair D.~Sullivan.}
        \and
        Daniel Kr\'al'\footnote{Mathematics Institute, DIMAP and Department of Computer Science, University of Warwick, Coventry CV4 7AL, UK. E-mail: {\tt d.kral@warwick.ac.uk}. The work of this author was supported by the European Research Council (ERC) under the European Union’s Horizon 2020 research and innovation programme (grant agreement No 648509). This publication reflects only its authors' view; the European Research Council Executive Agency is not responsible for any use that may be made of the information it contains.}
        \and
        Blair D.~Sullivan\footrecall{NCSU}
  }

\date{}
\maketitle

\begin{abstract}
    A graph is said to be walk-regular if, for each $\ell \geq 1$, every vertex is contained in the same number of closed walks of length $\ell$.
    We construct a $24$-vertex graph $H_4$ that is not walk-regular yet has maximized walk entropy,
    $S^V(H_4,\beta) = \log 24$, for some $\beta>0$.
    This graph is a counterexample to a conjecture of Benzi
    [Linear Algebra Appl.~443 (2014), 395--399, Conjecture 3.1].
    We also show that there exist infinitely many temperatures $\beta_0>0$ so that
    $S^V(G,\beta_0)=\log n_G$ if and only if a graph $G$ is walk-regular.\newline


\end{abstract}

\section{Introduction}

We study the interplay between the structural property of graphs called walk-regularity and
the algebraic property called walk entropy.
A simple graph $G$ is {\em walk-regular}~\cite{bib-godsil01+}
if every vertex of $G$ is contained in the same number of closed walks of length $\ell$ for every $\ell\in\NN$.
Observe that a graph $G$ is walk-regular if and only if
for every $\ell\in\NN$, all the diagonal entries of the power $\mA^\ell$ of the adjacency matrix $\mA$ of $G$ are the same.
Also note that if a graph $G$ is walk-regular, then it is necessarily degree-regular,
i.e., every vertex of $G$ has the same degree.

Estrada et al.~\cite{bib-estrada14+} initiated the study of the
relationship between walk-regularity and an algebraic parameter of a
graph called the walk entropy.
The {\em walk entropy} of a graph $G$ at the {\em temperature} $\beta\ge 0$ is defined as
$$S^V(G,\beta)=-\sum_{i=1}^{n_G} \frac{\left[\eexp^{\beta \mA}\right]_{ii}}{\Tr \eexp^{\beta \mA}} \log \frac{\left[\eexp^{\beta \mA}\right]_{ii}}{\Tr \eexp^{\beta \mA}}\;\mbox{,}$$
where $n_G$ denotes the number of vertices of $G$ (in general,
we use $n_H$ for the number of vertices of a graph $H$ throughout the paper).
In other words, the walk entropy is the entropy associated with the probability distribution on
the vertex set $V(G)$ that is linearly proportional to the subgraph centrality of the vertices.
We note that any probability distribution on $V(G)$ gives rise to a corresponding notion of graph entropy; Dehmer~\cite{dehmer2008information} called such distributions \emph{information functionals},
and introduced this more general class of graph entropies.

The \emph{subgraph centrality} of the $i$-th vertex of a graph $G$~\cite{bib-estrada05+}
is equal to $\left[\eexp^{\beta \mA}\right]_{ii}$, the corresponding diagonal entry of $\eexp^{\beta \mA}$.
Note that the walk entropy $S^V(G,\beta)\in [0,\log n_G]$ and
$S^V(G,\beta)=\log n_G$ if and only if all the diagonal entries of $\eexp^{\beta \mA}$ are the same.
That is, walk entropy is maximized precisely when all the vertices have the same subgraph centrality.

It is easy to see that if a graph $G$ is walk-regular,
then its walk entropy $S^V(G,\beta)$ is equal to $\log n_G$ for every $\beta\ge 0$.
Estrada et al.~\cite{bib-estrada14+} conjectured that the converse is also true.
\begin{conjecture}[{Estrada et al.~\cite[Conjecture 1]{bib-estrada14+}}]
\label{conj-estrada}
A graph $G$ is walk-regular if and only if $S^V(G,\beta)=\log n_G$ for all $\beta\geq0$.
\end{conjecture}
The conjecture was proven by Benzi in the following stronger form.
\begin{theorem}[{Benzi~\cite[Theorem 2.2]{bib-benzi14}}]
\label{thm-benzi}
Let $I$ be any set of real numbers containing an accumulation point.
If a graph $G$ satisfies $S^V(G,\beta)=\log n_G$ for all $\beta\in I$,
then $G$ is walk-regular.
\end{theorem}
Benzi also proposed the following strengthening of his result.
\begin{conjecture}[{Benzi~\cite[Conjecture 3.1]{bib-benzi14}}]
\label{conj-benzi}
A graph $G$ is walk-regular if and only if there exists $\beta>0$ such that $S^V(G,\beta)=\log n_G$.
\end{conjecture}
From the contrapositive, Estrada et al.~\cite{bib-estrada14+, bib-estrada14++} proposed that non--degree-regular graphs cannot have maximum walk entropy.
\begin{conjecture}[{Estrada et al.~\cite[Conjecture 1.2]{bib-estrada14++}}]
\label{conj-estrada-non-regular}
Let $G$ be a non--degree-regular graph. Then $S^V(G, \beta) < \log n_G$ for every $\beta > 0$.
\end{conjecture}
Estrada et al.~\cite{bib-estrada14++} attempted to prove Conjectures~\ref{conj-benzi}~and~\ref{conj-estrada-non-regular}, but their argument contains a flaw.
In Section~\ref{sect-counter}
we show that Conjectures~\ref{conj-benzi}~and~\ref{conj-estrada-non-regular} are false
by presenting a 24-vertex graph, which we denote $H_4$, that is not walk-regular yet attains $S^V(H_4,\beta)=\log 24$ for some $\beta>0$.
The graph $H_4$ contains vertices of degree four and five, i.e., it is not even degree-regular,
which resolves the question, mentioned in the concluding remarks in~\cite{bib-benzi14}, of whether degree-regularity is implied by the existence of a $\beta$ that maximizes walk entropy.
On the positive side, we show that there exist infinitely many temperatures $\beta_0>0$ such that a graph $G$ is walk-regular if and only if $S^V(G,\beta_0)=\log n_G$ (Corollary~\ref{cor-temp}),
i.e., there are temperatures that properly classify walk-regularity.

\section{Non--degree-regular graph maximizing walk entropy}
\label{sect-counter}

In this section, we present a counterexample to Conjectures~\ref{conj-benzi}
and~\ref{conj-estrada-non-regular}.
We start by presenting a closed formula for the diagonal entries of the exponential of the adjacency matrix of a graph.
Let $G$ be a graph and $\mA$ its adjacency matrix.
Further, let $\lambda_1,\ldots,\lambda_{n_G}$ be the eigenvalues of $\mA$ and
$\vu_1,\ldots,\vu_{n_G}$ an orthonormal basis formed by the eigenvectors of $\mA$.
It follows that
\begin{equation}
\left[\eexp^{\beta \mA}\right]_{ii}=\sum_{k=1}^{n_G} \vu_{k,i}^2\cdot\eexp^{\beta\lambda_k}
\label{eq-diag}
\end{equation}
for every $i=1,\ldots,n_G$.
In particular, each diagonal entry of $\eexp^{\beta \mA}$ is an analytic function of $\beta$,
which is a linear combination of at most $n_G$ exponential functions.

We next present a construction of graphs $H_m$ parameterized by a positive integer $m\in\NN$.
The graph $H_m$ is obtained from $m$ isolated vertices and $m+1$ cliques of order $m$
by including a perfect matching between the $m$ isolated vertices and each of the $m+1$ cliques.
The graph $H_m$ has $m+(m+1)m=m^2+2m$ vertices; $m$ vertices have degree $m+1$ and
the remaining $m^2+m$ vertices have degree $m$.
We are convinced that $H_m$ is a counterexample to Conjecture~\ref{conj-benzi} for~every $m\ge 4$;
however, we will here analyze the case $m=4$ only.
The graph $H_4$ and its adjacency matrix are presented in Figure~\ref{fig-G24}.

\begin{figure}
\begin{center}
\epsfbox{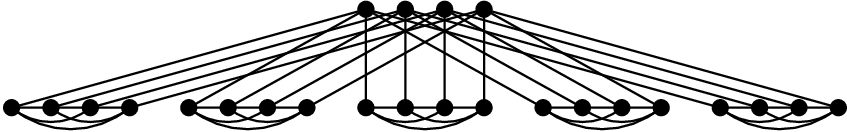}\\
\vspace{10pt}
\scalebox{0.82}{
$\matbd{A}_{K_4} = \left(
\begin{array}{cccc}
  0 & 1 & 1 & 1 \\
  1 & 0 & 1 & 1 \\
  1 & 1 & 0 & 1 \\
  1 & 1 & 1 & 0
\end{array}
\right)$,
$\matbd{I}_4 = \left(
\begin{array}{cccc}
  1 & 0 & 0 & 0 \\
  0 & 1 & 0 & 0 \\
  0 & 0 & 1 & 0 \\
  0 & 0 & 0 & 1
\end{array}
\right)$,
$\mA_{H_4} = \left(
\begin{array}{cccccc}
0 & \matbd{I}_4 & \matbd{I}_4 & \matbd{I}_4 & \matbd{I}_4 & \matbd{I}_4 \\
\matbd{I}_4 & \mA_{K_4} & 0 & 0 & 0 & 0 \\
\matbd{I}_4 & 0 & \mA_{K_4} & 0 & 0 & 0 \\
\matbd{I}_4 & 0 & 0 & \mA_{K_4} & 0 & 0 \\
\matbd{I}_4 & 0 & 0 & 0 & \mA_{K_4} & 0 \\
\matbd{I}_4 & 0 & 0 & 0 & 0 & \mA_{K_4}
\end{array}
\right)$
}
\end{center}
\caption{The graph $H_4$ and its adjacency matrix $\mA_{H_4}$.}
\label{fig-G24}\label{tab-adj}
\end{figure}

\begin{theorem}\label{thm-counter}
There exists $\beta>0$ such that $S^V(H_4,\beta)=\log 24$.
\end{theorem}

\begin{proof}
Let $\mA$ be the adjacency matrix of the graph $H_4$.
The matrix $\mA$ has six different eigenvalues,
which are given with corresponding eigenvectors in Table~\ref{tab-eigen}.
Note that the vectors given in Table~\ref{tab-eigen} are not normalized to be unit and orthogonal.
We will show that there exists $\beta>0$ such that
all the diagonal entries of the matrix $\eexp^{\beta \mA}$ are the same,
which yields the statement of the theorem.

\begin{table}
\begin{center}
  \scalebox{0.8}{
    $\begin{array}{cc}
    \mbox{Eigenvalue} & \mbox{ \hspace{178pt} Eigenvector \hspace{178pt}} \\
    \end{array}$
  }
\scalebox{0.6}{
$\begin{array}{c@{\hskip 0.05in}lllllllr}
\toprule
\vspace{2pt}
\frac{3+\sqrt{29}}{2}\approx 4.193 &
\Big( \hspace{3pt} \frac{10}{3+\sqrt{29}}, \frac{10}{3+\sqrt{29}}, \frac{10}{3+\sqrt{29}}, \frac{10}{3+\sqrt{29}}, &
1, 1, 1, 1, & 1, 1, 1, 1, & 1, 1, 1, 1, & 1, 1, 1, 1, & 1, 1, 1, 1 & \Big) \\ \vspace{2pt}

3 &
\Big( \hspace{3pt}  0, 0, 0, 0, &
1, 1, 1, 1, & -1, -1, -1, -1, & 0, 0, 0, 0, & 0, 0, 0, 0, & 0, 0, 0, 0 & \Big) \\ \vspace{2pt}
3 &
\Big( \hspace{3pt}  0, 0, 0, 0, &
1, 1, 1, 1, & 0, 0, 0, 0, & -1, -1, -1, -1, & 0, 0, 0, 0, & 0, 0, 0, 0 & \Big) \\ \vspace{2pt}
3 &
\Big( \hspace{3pt}  0, 0, 0, 0, &
1, 1, 1, 1, & 0, 0, 0, 0, & 0, 0, 0, 0, & -1, -1, -1, -1, & 0, 0, 0, 0 & \Big) \\ \vspace{2pt}
3 &
\Big( \hspace{3pt}  0, 0, 0, 0, &
1, 1, 1, 1, & 0, 0, 0, 0, & 0, 0, 0, 0, & 0, 0, 0, 0, & -1, -1, -1, -1 & \Big) \\ \vspace{2pt}

\frac{-1+\sqrt{21}}{2}\approx 1.791 &
\Big( \hspace{3pt}  \frac{10}{-1+\sqrt{21}}, \frac{-10}{-1+\sqrt{21}}, 0, 0, &
1, -1, 0, 0, & 1, -1, 0, 0, & 1, -1, 0, 0, & 1, -1, 0, 0, & 1, -1, 0, 0 & \Big) \\ \vspace{2pt}
\frac{-1+\sqrt{21}}{2}\approx 1.791 &
\Big( \hspace{3pt}  \frac{10}{-1+\sqrt{21}}, 0, \frac{-10}{-1+\sqrt{21}}, 0, &
1, 0, -1, 0, & 1, 0, -1, 0, & 1, 0, -1, 0, & 1, 0, -1, 0, & 1, 0, -1, 0 & \Big) \\ \vspace{2pt}
\frac{-1+\sqrt{21}}{2}\approx 1.791 &
\Big( \hspace{3pt}  \frac{10}{-1+\sqrt{21}}, 0, 0, \frac{-10}{-1+\sqrt{21}}, &
1, 0, 0, -1, & 1, 0, 0, -1, & 1, 0, 0, -1, & 1, 0, 0, -1, & 1, 0, 0, -1 & \Big) \\ \vspace{2pt}

-1 &
\Big( \hspace{3pt}  0, 0, 0, 0, &
1, -1, 0, 0, & -1, 1, 0, 0, & 0, 0, 0, 0, & 0, 0, 0, 0, & 0, 0, 0, 0 & \Big) \\ \vspace{2pt}
-1 &
\Big( \hspace{3pt}  0, 0, 0, 0, &
1, 0, -1, 0, & -1, 0, 1, 0, & 0, 0, 0, 0, & 0, 0, 0, 0, & 0, 0, 0, 0 & \Big) \\ \vspace{2pt}
-1 &
\Big( \hspace{3pt}  0, 0, 0, 0, &
1, 0, 0, -1, & -1, 0, 0, 1, & 0, 0, 0, 0, & 0, 0, 0, 0, & 0, 0, 0, 0 & \Big) \\ \vspace{2pt}
-1 &
\Big( \hspace{3pt}  0, 0, 0, 0, &
1, -1, 0, 0, & 0, 0, 0, 0, & -1, 1, 0, 0, & 0, 0, 0, 0, & 0, 0, 0, 0 & \Big) \\ \vspace{2pt}
-1 &
\Big( \hspace{3pt}  0, 0, 0, 0, &
1, 0, -1, 0, & 0, 0, 0, 0, & -1, 0, 1, 0, & 0, 0, 0, 0, & 0, 0, 0, 0 & \Big) \\ \vspace{2pt}
-1 &
\Big( \hspace{3pt}  0, 0, 0, 0, &
1, 0, 0, -1, & 0, 0, 0, 0, & -1, 0, 0, 1, & 0, 0, 0, 0, & 0, 0, 0, 0 & \Big) \\ \vspace{2pt}
-1 &
\Big( \hspace{3pt}  0, 0, 0, 0, &
1, -1, 0, 0, & 0, 0, 0, 0, & 0, 0, 0, 0, & -1, 1, 0, 0, & 0, 0, 0, 0 & \Big) \\ \vspace{2pt}
-1 &
\Big( \hspace{3pt}  0, 0, 0, 0, &
1, 0, -1, 0, & 0, 0, 0, 0, & 0, 0, 0, 0, & -1, 0, 1, 0, & 0, 0, 0, 0 & \Big) \\ \vspace{2pt}
-1 &
\Big( \hspace{3pt}  0, 0, 0, 0, &
1, 0, 0, -1, & 0, 0, 0, 0, & 0, 0, 0, 0, & -1, 0, 0, 1, & 0, 0, 0, 0 & \Big) \\ \vspace{2pt}
-1 &
\Big( \hspace{3pt}  0, 0, 0, 0, &
1, -1, 0, 0, & 0, 0, 0, 0, & 0, 0, 0, 0, & 0, 0, 0, 0, & -1, 1, 0, 0 & \Big) \\ \vspace{2pt}
-1 &
\Big( \hspace{3pt}  0, 0, 0, 0, &
1, 0, -1, 0, & 0, 0, 0, 0, & 0, 0, 0, 0, & 0, 0, 0, 0, & -1, 0, 1, 0 & \Big) \\ \vspace{2pt}
-1 &
\Big( \hspace{3pt}  0, 0, 0, 0, &
1, 0, 0, -1, & 0, 0, 0, 0, & 0, 0, 0, 0, & 0, 0, 0, 0, & -1, 0, 0, 1 & \Big) \\ \vspace{2pt}

\frac{3-\sqrt{29}}{2}\approx -1.193 &
\Big( \hspace{3pt}  \frac{10}{3-\sqrt{29}}, \frac{10}{3-\sqrt{29}}, \frac{10}{3-\sqrt{29}}, \frac{10}{3-\sqrt{29}}, &
1, 1, 1, 1, & 1, 1, 1, 1, & 1, 1, 1, 1, & 1, 1, 1, 1, & 1, 1, 1, 1 & \Big) \\ \vspace{2pt}

\frac{-1-\sqrt{21}}{2}\approx -2.791 &
\Big( \hspace{3pt}  \frac{10}{-1-\sqrt{21}}, \frac{-10}{-1-\sqrt{21}}, 0, 0, &
1, -1, 0, 0, & 1, -1, 0, 0, & 1, -1, 0, 0, & 1, -1, 0, 0, & 1, -1, 0, 0 & \Big) \\ \vspace{2pt}
\frac{-1-\sqrt{21}}{2}\approx -2.791 &
\Big( \hspace{3pt}  \frac{10}{-1-\sqrt{21}}, 0, \frac{-10}{-1-\sqrt{21}}, 0, &
1, 0, -1, 0, & 1, 0, -1, 0, & 1, 0, -1, 0, & 1, 0, -1, 0, & 1, 0, -1, 0 & \Big) \\ \vspace{2pt}
\frac{-1-\sqrt{21}}{2}\approx -2.791 &
\Big( \hspace{3pt}  \frac{10}{-1-\sqrt{21}}, 0, 0, \frac{-10}{-1-\sqrt{21}}, &
1, 0, 0, -1, & 1, 0, 0, -1, & 1, 0, 0, -1, & 1, 0, 0, -1, & 1, 0, 0, -1 & \Big)

\end{array}
$}
\end{center}
\caption{The eigenvalues and the corresponding eigenvectors of the adjacency matrix of the graph $H_4$.}
\label{tab-eigen}
\end{table}

Let $f_1(\beta)=\left[\eexp^{\beta \mA}\right]_{11}$ and $f_2(\beta)=\left[\eexp^{\beta \mA}\right]_{55}$.
Note that $\left[\eexp^{\beta \mA}\right]_{ii}=f_1(\beta)$ for $i=1,\ldots,4$ and
$\left[\eexp^{\beta \mA}\right]_{ii}=f_2(\beta)$ for $i=5,\ldots,24$.
Observe that the $k$-th derivative of $f_i(\beta)$ for $\beta=0$ is equal
to the corresponding diagonal entry of $\mA^k$.
In particular, $f_1(0)=f_2(0)=1$, $f'_1(0)=f'_2(0)=0$, $f''_1(0)=5$ and $f''_2(0)=4$.
This implies that there exists $\varepsilon>0$ such that $f_1(\beta)>f_2(\beta)$ for all $\beta\in (0,\varepsilon)$.
On the other hand, it holds that $f_1(1)<f_2(1)$; a direct computation shows that
$f_1(1)\approx 6.481$ and $f_2(1)\approx 7.175$.
Since both $f_1(\beta)$ and $f_2(\beta)$ are continuous functions of $\beta$ on the interval $[0,1]$,
there exists $\beta\in (0,1)$ such that $f_1(\beta)=f_2(\beta)$,
i.e., such that all the diagonal entries of $\eexp^{\beta \mA}$ are the same.
\end{proof}

A numerical computation yields that the value of $\beta$
from the proof of Theorem~\ref{thm-counter} is approximately $0.499$.
There also exists $\beta>1$ such that $f_1(\beta)=f_2(\beta)$.
Indeed, assume that $\lambda_1$ is the largest eigenvalue of the adjacency matrix of the graph $H_4$ and
let $\vu_1$ be the corresponding unit eigenvector.
Since $\vu_{1,1}^2>\vu_{1,5}^2$, it follows from (\ref{eq-diag}) that
there exists $\beta_0>0$ such that $f_1(\beta)>f_2(\beta)$ for all $\beta\ge\beta_0$.
Since $f_1(1)<f_2(1)$, we get that there exists $\beta\in (1,\beta_0)$ such that $f_1(\beta)=f_2(\beta)$;
a numerical computation shows that the value of such $\beta$ is approximately $1.912$.
We remark that each of the graphs $H_m$ with $m\in\NN$ greater than 1, has six different eigenvalues
with the structure of corresponding eigenvectors similar to that of $H_4$.

We would like to conclude with an intuitive explanation behind the construction of the graph $H_m$.
The graph $H_m$ has vertices of degree $m$ and $m+1$, and
any pair of vertices of the same degree can be mapped to each other by an automorphism of $H_m$.
The values of the diagonal entries of $\eexp^{\beta\mA}$
are controlled by the diagonal entries of $\mA^2$ for $\beta$ close to zero,
by the diagonal entries of $\mA^3$ for larger (but still small) values of $\beta$,
then by the diagonal entries of $\mA^4$, etc.
As $k$ grows, the diagonal of $\mA^k$ becomes proportional to the Perron-Frobenius eigenvector of $\mA$.
Hence, the diagonal entries of $\eexp^{\beta \mA}$ are proportional to the degrees of the corresponding vertices for the first regime of $\beta$, and to the eigenvector centrality of the vertices for $\beta$ in the third regime (for a precise analysis, see~\cite{benziklymko-2015}).
In the graph $H_m$,
the vertex degrees and Perron-Frobenius eigenvector values produce the same ordering on the vertices.
On the other hand, the diagonal entries of $\eexp^{\beta \mA}$ corresponding to the vertices of degree $m$ become larger than those corresponding to the vertices of degree $m+1$ in the middle regime of $\beta$,
since the vertices of degree $m$ are contained in many triangles and cycles of length four.
This explains the behavior of the functions $f_1$ and $f_2$ that
we have observed in the proof of Theorem~\ref{thm-counter}.

\section{Temperatures classifying walk-regularity}
\label{sec-temp}

We start by observing that
a graph $G$ achieves the maximum walk entropy for at most finitely many temperatures
unless $G$ is walk-regular.

\begin{theorem}
\label{thm-temp}
If a graph $G$ is not walk-regular, then there are only finitely many $\beta>0$ such that $S^V(G,\beta)=\log n_G$.
\end{theorem}

\begin{proof}
We proceed by contradiction.
Suppose that there exists a graph $G$ that is not walk-regular
but the set $I$ of $\beta\ge 0$ such that $S^V(G,\beta)=\log n_G$ is infinite.
Since the diagonal entries of $\eexp^{\beta \mA}$ are continuous functions of $\beta$ bounded away from zero,
it follows that $S^V(G,\beta)$ is continuous and the set $I$ is closed.
Let $\mA$ be the adjacency matrix of $G$,
$\lambda_1,\ldots,\lambda_k$ all distinct eigenvalues of $\mA$, and
$f_i(\beta)$ the $i$-th diagonal entry of $\eexp^{\beta \mA}$, $i=1,\ldots,n_G$.
We can assume that $\lambda_1>\cdots>\lambda_k$.
By (\ref{eq-diag}), there exist non-negative reals $a_{ij}$, $i=1,\ldots,n_G$ and $j=1,\ldots,k$ such that
\begin{equation}
f_i(\beta)=\sum_{j=1}^k a_{ij}\cdot\eexp^{\beta\lambda_j}\label{eq-temp}
\end{equation}
for every $i=1,\ldots,n_G$ and $\beta\ge 0$.
If $f_1(\beta)=\cdots=f_{n_G}(\beta)$ for all $\beta>0$, $G$ would be walk-regular by Theorem~\ref{thm-benzi}.
Hence, $f_i\not=f_{i'}$ for some $i\not=i'$.  By symmetry, we can assume that $f_1\not=f_2$.

Let $j$ be the smallest integer such that $a_{1j}\not=a_{2j}$.
We can assume by symmetry that $a_{1j}>a_{2j}$, which implies that
$$\lim_{\beta\to\infty} \left(f_1(\beta)-f_2(\beta)\right)=\infty\;\mbox{.}$$
It follows that there exists $\beta_0$ such that $f_1(\beta)\not=f_2(\beta)$ for all $\beta\ge\beta_0$,
i.e., the set $I$ is a subset of the interval $[0,\beta_0)$.
Since the set $I$ is infinite,
it has an accumulation point,
which implies that $G$ is walk-regular by Theorem~\ref{thm-benzi}, contrary to our original assumption.
\end{proof}

The next corollary immediately follows from Theorem~\ref{thm-temp}.

\begin{corollary}\label{cor-temp}
There exists $\beta_0>0$ such that the following holds:
a graph $G$ is walk-regular if and only if $S^V(G,\beta_0)=\log n_G$.
\end{corollary}

\begin{proof}
Let $X$ be the set of all $\beta>0$ such that
there exists a graph $G$ that is not walk-regular and $S^V(G,\beta)=\log n_G$.
Since there are only finitely many such $\beta$ for each non--walk-regular graph $G$ by Theorem~\ref{thm-temp},
the set $X$ is countable.
Hence, there exists $\beta_0\in (0,\infty)\setminus X$ and
any such $\beta_0$ has the property claimed in the statement of the corollary.
\end{proof}

\section{Concluding remarks}

Corollary~\ref{cor-temp} shows that there are temperatures $\beta_0>0$ such that,
for any graph $G$, the graph $G$ is walk-regular if and only if its walk entropy for $\beta_0$ is $\log n_G$.
Unfortunately, we were not able to explicitly find any such $\beta_0$, and so it remains an open problem to identify a value $\beta_0$ with this property.
It could be the case that $\beta_0=1$ is such a value of interest, as
conjectured by Estrada~\cite{bib-estrada13}.

\begin{conjecture}[Estrada~\cite{bib-estrada13}, Conjecture 3]
A graph $G$ is walk-regular if and only if $S^V(G,1)=\log n_G$.
\end{conjecture}

It could even be the case that for every non--walk-regular $G$, the walk entropy $S^V(G,\beta)$ is not maximized at any positive, rational value $\beta_0$.

\begin{conjecture}
A graph $G$ is walk-regular if and only if there exists a rational $\beta>0$ such that $S^V(G,\beta)=\log n_G$.
\end{conjecture}

Theorem~\ref{thm-temp} asserts that
if a graph $G$ is not walk-regular,
then the set of temperatures $\beta$ such that $S^V(G,\beta)=\log n_G$ is finite.
We believe that it is possible to bound the size of this set in terms of the number of vertices of $G$ as follows.

\begin{conjecture}
If a graph $G$ is not walk-regular, then there are at most $n_G-1$ values $\beta>0$ such that $S^V(G,\beta)=\log n_G$.
\end{conjecture}

\section*{Acknowledgments}
The authors would like to thank Eric Horton for useful discussions on non-walk-regular graphs that maximize other notions of graph entropy.

\end{document}